\documentclass{amsart}
\usepackage{amssymb}
\usepackage{pb-diagram}

\title[Cohomology of free loop spaces]
{
Cohomology of classifying spaces of loop groups and finite Chevalley groups 
associated with spin groups
}
\author
{
Masaki Kameko
}
\address
{
Department of Mathematical Sciences,
Shibaura Institute of Technology,
307 Minuma-ku Fukasaku, Saitama-City 337-8570, Japan
}
\email
{
kameko@shibaura-it.ac.jp
}
\thanks
{
The author is partially supported by the Japan Society for the Promotion of Science, 
Grant-in-Aid for Scientific Research (C) 22540102, 25400097.
}

\newtheorem{theorem}{Theorem}[section]
\newtheorem{proposition}{Proposition}[section]

\newtheorem{conjecture}{Conjecture}[section]

\theoremstyle{definition}

\makeatletter
\let\c@proposition=\c@theorem
\let\c@lemma=\c@theorem
\let\c@corollary=\c@theorem
\let\c@conjecture=\c@theorem
\let\c@definition=\c@theorem
\let\c@remark=\c@theorem
\let\c@assumption=\c@theorem
\makeatother

\usepackage[usenames]{color}

\newcommand{\Loop}{\mathcal{L}}
\newcommand{\spin}{\mathrm{Spin}}
\newcommand{\sq}{\mathrm{Sq}}
\newcommand{\Bousfield}[1]
{ (\mathbb{Z}/\ell)_\infty #1 }
\newcommand{\Friedlander}[1]
{\mathop{\mathrm{holim}}_{\longleftarrow} \Bousfield{(B#1)}_{et}}
\newcommand{\Closure}{\overline{\mathbb{F}}_p}

\newcommand{\Spec}[1]{\mathrm{Spec}(#1)}
\newcommand{\HyperRigidCover}[1]{\mathrm{HRR}(#1)}

\begin{document}

\begin{abstract}

Let  $q$ be a power of an odd prime $p$ and $\mathbb{F}_q$ 
the finite field with $q$ elements. 
For $n\geq 3$, we prove the mod $2$ cohomology  of the finite Chevalley group 
$\spin_n(\mathbb{F}_q)$
is isomorphic to that of the classifying space of the loop group of the spin group 
$\spin(n)$.
\end{abstract}

\maketitle

\section{Introduction}

In this paper, we reduce the computation of the mod $2$ 
cohomology of certain homotopy fixed 
point sets associated with the spin group $\spin(n)$ to the computation of  Gysin sequences. 
As an  application of such reduction, we prove the following theorem.

\begin{theorem}\label{theorem:main1}
Suppose that $n \geq 3$.
Let $q$ be a power of an odd prime $p$.
Let $\mathbb{F}_q$ be the finite group with $q$ elements.
As a graded $\mathbb{Z}/2$-module, 
the mod $2$ cohomology of the  finite Chevalley group $\spin_{n}(\mathbb{F}_q)$
is isomorphic to 
that of the free loop space of the classifying space of  the spin group 
$\spin(n)$.
\end{theorem}

We denote by $\Loop X$ the free loop space of $X$, that is, 
the space of continuous maps from the circle $S^1$ to $X$, 
so that $\Loop X=\mathrm{Map}(S^1, X)$. For a topological group $G$, 
we have the loop group $\Loop G$. 
First, we recall the relation between the cohomology of the classifying space
$B\Loop G$ 
and that of 
homotopy fixed point sets. 
We recall the following proposition, that  is Proposition~{2.4} of Atiyah
and Bott \cite{atiyah-bott-1982}.
\begin{proposition}[Atiyah-Bott]\label{proposition:atiyah-bott}
Let $G$ be a connected Lie group.
Let $P$ be a principal $G$-bundle over a manifold $M$ and denote 
the gauge group by $\mathcal{G}(P)$.
Then, $B\mathcal{G}(P)\simeq \mathrm{Map}_{P}(M, BG)$, 
where $\mathrm{Map}_{P}(M, BG)$ is the 
subset of $\mathrm{Map}(M, BG)$ whose pull-back of the universal 
$G$-bundle is isomorphic to $P$.
\end{proposition}

Considering the case that $P$ is the trivial bundle over $S^1$, it is easy to see that
$\mathrm{Map}_P(S^1, BG)=\mathrm{Map}(S^1, BG)=\Loop BG$ and 
$\mathcal{G}(P)=\mathit{\Gamma}(\mathrm{Ad} P)=\Loop G$. Thus, we  have 
$B\Loop G\simeq \Loop BG$ for a connected Lie group $G$.
We recall the definition of the homotopy fixed point set of a map $f:X\to X$, 
which we denote by $\Loop_f X$. The homotopy fixed point set $\Loop_f X$ of a map $f:X\to X$
is defined by the following 
pull-back diagram:
\[
\begin{diagram}
\node{\Loop_f X} \arrow{e} \arrow{s,l}{\pi} \node{\mathrm{Map}(I,X)} \arrow{s,r}{\varphi} \\
\node {X} \arrow{e,t}{1\times f} \node{X\times X,}
\end{diagram}
\]
where $\mathrm{Map}(I,X)$ is the space of continuous maps from the unit interval 
$I=[0,1]$ to $X$, $\varphi(\lambda)=(\lambda(0), \lambda(1))$ 
and, by abuse of notation,  we denote the composition of the map 
$1\times f$ and the diagonal 
map of $X$ by $1\times f$.
The space $\Loop_f X$ is a subspace of $\mathrm{Map}(I, X)$
and the projection $\pi$ is the restriction of the evaluation at $0$, 
$\mathrm{Map}(I, X) \to X$,  to the subspace $\Loop_f X$ , that is, $\pi(\lambda)=\lambda(0)$.
It is clear that the free loop space $\Loop X$ is 
the homotopy fixed point set of the identity 
map of $X$.
Thus, the computation of the cohomology of the classifying space of the loop group 
$\Loop G$ is 
nothing but that of the free loop space of $BG$, hence 
it is that of the homotopy fixed point set of the identity map of $BG$.

Next, we recall the relation between the mod $\ell$ cohomology 
of finite Chevalley groups and that of homotopy fixed point sets, 
where $\ell$ is a prime number.
Let $p$ be a prime number district from $\ell$, $q$ a power of 
$p$ and $\mathbb{F}_q$ the finite field with $q$ elements. 
We write $\Closure$ for the algebraic closure of $\mathbb{F}_q$.
The finite Chevalley group $G(\mathbb{F}_q)$ is obtained as the fixed point set of  
the Frobenius map
\[
\phi^q:G(\Closure)\to G(\Closure),
\]
induced by the Frobenius map $\phi^q:\Closure \to \Closure$ sending $x$ to $x^q$,
where $G(\Closure)$ is the geometric points of 
a reductive group scheme over $\Closure$ obtained  
by taking a Chevalley basis for 
the complexification  of a compact connected Lie group $G$ and,
by abuse of notation, we denote by the same symbol $G$ the above 
reductive group scheme over $\Closure$ obtained from 
a compact connected Lie group $G$.
Quillen computes the mod $\ell$ cohomology of the finite general linear group 
$GL_n(\mathbb{F}_q)$ in \cite{quillen-1972}
by computing that of a homotopy fixed point set.
Friedlander expands Quillen's theory and shows that,
in general, the mod $\ell$ cohomology of the finite Chevalley group 
$G(\mathbb{F}_q)$ associated 
with a compact connected Lie group $G$ is  that of the homotopy 
fixed point set of  the Frobenius map 
\[
\phi^{q}:\left|\Friedlander{G}\right| \to \left|\Friedlander{G}\right|,
\]
where the Frobenius map $\phi^q$ is induced by 
the Frobenius map $\phi^q:\Closure \to \Closure$ above, 
$\Bousfield X$ is the Bousfield-Kan $\mathbb{Z}/\ell$-completion of a simplicial set $X$
and $|X|$ is the geometric realization of a simplicial set $X$.
For details of Quillen and Friedlander's theory, we refer the reader to 
Friedlander \cite{friedlander-1982}.
In particular, by Proposition~{8.8} and Theorem~{12.2} 
in Friedlander~\cite{friedlander-1982}, we have the following theorem.
\begin{theorem}[Friedlander]\label{theorem:friedlander}
The following holds\,{\rm :}
\begin{itemize}
\item[(1)] $\displaystyle \left|\Friedlander{G}\right|$ is homotopy equivalent to 
$\left|\Bousfield \mathrm{Sing}_\bullet (BG)\right|$ where $\mathrm{Sing}_\bullet(BG)$
 is the simplicial set obtained from 
the classifying space of the connected Lie group $G$ by applying the singular functor
and
\item[(2)] $H^{*}(BG(\mathbb{F}_q);\mathbb{Z}/\ell )
=H^{*}(\Loop_{\phi^q} \displaystyle 
\left|\Friedlander{G}\right|;\mathbb{Z}/\ell )$.
\end{itemize}
\end{theorem}

Thus, to prove Theorem~\ref{theorem:main1}, 
it suffices to compute  the mod
$2$ cohomology of homotopy fixed point sets of self-maps of 
a space $B\spin_n$, defined by
\[
B\spin_n=\left|\mathop{\mathrm{holim}}_{\longleftarrow} (\mathbb{Z}/2)_{\infty}{(B\spin(n))}_{et}\right|,
\]
where  we fix an odd prime number $p$ and 
$\spin(n)$ in the right-hand side is a reductive group scheme over $\Closure$
associated with the spin group $\spin(n)$.
Suppose that  $q$ is a power of $p$. Then, we have the Frobenius map $\phi^q:B\spin_n \to B\spin_n$.
Theorem~\ref{theorem:main1} follows from the following theorem.

\begin{theorem} \label{theorem:main2}
Let $q$ be a power of $p$, $\phi^q:B\spin_n \to B\spin_n$ the Frobenius map.
As a graded $\mathbb{Z}/2$-module, the mod $2$ cohomology of $\Loop_{\phi^q} B\spin_n$ is isomorphic to that of $\Loop B\spin_n$.
\end{theorem}


Despite the fact that the loop group $\Loop G$ and finite Chevalley groups
$G(\mathbb{F}_q)$ are different form each other either as topological spaces
 or as groups, Tezuka conjectures the following in \cite{tezuka-1998}.
\begin{conjecture}[Tezuka]\label{conjecture:tezuka}
Let $\ell$ be a prime number such that $\ell$ (reps., $4$) divides $q-1$ if $\ell$
 is odd (reps., even). Then, we have a ring isomorphism \[
H^{*}(\Loop BG; \mathbb{Z}/\ell) \cong H^{*}(BG(\mathbb{F}_q); \mathbb{Z}/\ell).
\]
\end{conjecture}


Our main theorem, Theorem~\ref{theorem:main1}, 
 is  a weak form of Tezuka's conjecture for 
$\ell=2, G=\spin(n)$, $n\geq3$ in the sense that the isomorphism in Theorem~\ref{theorem:main1} is not
a ring isomorphism but an isomorphism of graded 
$\mathbb{Z}/2$-modules only.
In the course of our proof of Theorem~\ref{theorem:main2}, we prove 
Theorem~\ref{theorem:SO} below. 
It is a strong form of Tezuka's conjecture for $\ell=2, G=SO(n)$ 
in the sense that we have an isomorphism of algebras over the 
mod $2$ Steenrod algebra 
instead of a ring isomorphism. 

To state Theorem~\ref{theorem:SO}, we need to set up notations and definitions.
Throughout the rest of this paper, since we deal with the mod $2$ cohomology only, 
we denote the mod $2$ cohomology of a space 
$X$ by $H^{*}(X)$.
Also, we assume that $q$ is a power of an odd prime $p$.
We denote by $\mathbb{F}_q$ the finite filed with $q$ elements and by 
$\Closure$ its algebraic closure.
For the sake of notational simplicity, let 
\[
BSO_n=\left|\mathop{\mathrm{holim}}_{\longleftarrow} 
(\mathbb{Z}/2)_\infty{(BSO(n))}_{et}\right|,
\]
where $SO(n)$ in the right-hand side
is a reductive group scheme associated with the special orthogonal group $SO(n)$.

Let $A_{n-1}$ be a subgroup of the connected Lie group $SO(n)$ consisting of diagonal matrixes 
whose diagonal entries are $\pm 1$. 
Denote by $S_n$ the Weyl group of $A_{n-1}$, 
that is, the quotient of the normalizer of $A_{n-1}$ by the centralizer of $A_{n-1}$.
For details of Weyl groups,  we refer the reader to Chapter 3, \S4 in the book of Mimura and Toda
 \cite{mimura-toda-book}, in particular, p.135, Corollary 4.19.
The Weyl group of $A_{n-1}$ is isomorphic to the symmetric group on
$n$ letters. So, we denote it by $S_n$. One may consider
$A_{n-1}$ as a subgroup  of $SO_n(\mathbb{F}_p)$ 
where $\mathbb{F}_p$ is the prime field of characteristic $p\not =2$. 
We may also consider the Weyl group as a sub-quotient of 
$SO_n(\mathbb{F}_p)$. 
We need the following proposition to state Theorem~\ref{theorem:SO}.

\begin{proposition}
\label{proposition:ba}
There exists the classifying space $BA_{n-1}$ of $A_{n-1}$ with  an action of the Weyl group $S_n$ on $BA_{n-1}$ and a map 
\[
\alpha_n:BA_{n-1}\to 
BSO_n
\]
satisfying the following equalities in the category of
 topological spaces and continuous maps:\begin{itemize}
\item[(1)] $\phi^q \circ \alpha_n=\alpha_n$, 
\item[(2)]  
$\alpha_n \circ g=\alpha_n$ for  each $g \in S_n$.
\end{itemize}
\end{proposition}

We will prove this proposition in \S2. 
We emphasize that the  equalities in the above proposition do not mean
{\it up to homotopy}
but they are equalities in the category of topological spaces.

Restricting the evaluation at $0$, $\mathrm{Map}(I, BSO_n) \to BSO_n$,
to subspaces $\Loop BSO_n$, $\Loop_{\phi^q} BSO_n$, 
we obtain projections 
\[
\pi_n:\Loop BSO_n \to BSO_n, \quad \pi_{n,q}:\Loop_{\phi^q} BSO_n \to BSO_n.
\]
Let us define a space $\tilde{B}A_{n-1}$ and maps 
$\tilde{\alpha}_n:\tilde{B}A_{n-1} \to \Loop BSO_n$, 
$\tilde{\alpha}_{n,q}:\tilde{B}A_{n-1} \to \Loop_{\phi^q}BSO_n$
by the following pull-back diagram:
\[
\begin{diagram}
\node{\Loop BSO_n} \arrow{s,l}{\pi_n}  \node{\tilde{B}A_{n-1}} 
\arrow{s,r}{\pi_0} 
\arrow{e,t}{\tilde{\alpha}_{n, q}} 
\arrow{w,t}{\tilde{\alpha}_n} 
\node{\Loop_{\phi^q}BSO_n} \arrow{s,r}{\pi_{n,q}}
\\
\node{BSO_n} 
\node{BA_{n-1}} 
\arrow{e,t}{\alpha_n} 
\arrow{w,t}{\alpha_n} \node{BSO_n.}
\end{diagram}
\]
Since $\phi^q\circ \alpha_n=\alpha_n$, the space $\tilde{B}A_{n-1}$ and the map 
$\pi_0$ above are well-defined. 
Moreover, the Weyl group 
$S_n$ acts on $BA_{n-1}$.
Recall that $\alpha_n\circ g=\alpha_n$ for all $g$ in $S_n$.
For $(x, \lambda) \in \tilde{B}A_{n-1}$, consider $(gx, \lambda)$. 
Since
 $\alpha_n(gx)=\alpha_n(x)=\lambda(0)=\lambda(1)$, $(gx, \lambda)$ is also an element in $\tilde{B}A_{n-1}$.
Thus, we have a  map $g:\tilde{B}A_{n-1} \to \tilde{B}A_{n-1}$ sending $(x, \lambda)$ to $(gx, \lambda)$
and an action of $S_n$ on $\tilde{B}A_{n-1}$. We also have
$\tilde{\alpha}_{n,q} \circ g=\tilde{\alpha}_{n,q}$ and 
$\tilde{\alpha}_{n} \circ g=\tilde{\alpha}_n$.
Therefore, the images of the induced homomorphisms 
$\tilde{\alpha}_{n,q}^*$, $\tilde{\alpha}_n^*$ 
are in $H^{*}(\tilde{B}A_{n-1})^{S_n}$
which is the ring of invariants of $H^{*}(\tilde{B}A_{n-1})$ with respect to the action of 
$S_n$, that is, 
\[
H^{*}(\tilde{B}A_{n-1})^{S_n}=
\{ x\in H^{*}(\tilde{B}A_{n-1}) \;|\; g^*x=x \; \mbox{ for all $g \in S_n$} \}.
\]
Then, our version of the  strong form of Tezuka's conjecture
for $\ell=2$, $G=BSO(n)$ is as follows:


\begin{theorem}\label{theorem:SO}
The induced homomorphisms \[
H^{*}(\Loop BSO_n)\stackrel{\tilde{\alpha}_n^*}{\longrightarrow}  H^{*}(\tilde{B}A_{n-1})^{S_n}
\stackrel{ \tilde{\alpha}_{n,q}^*}{\longleftarrow}
H^{*}(\Loop_{\phi^q} BSO_n)
 \] are isomorphisms.
\end{theorem}

It follows immediately  from Theorem~\ref{theorem:SO} that 
$H^{*}(\Loop_{\phi^q} BSO_n)$ 
is isomorphic to $H^{*}(\Loop BSO_n)$ 
as an algebra over the mod $2$ Steenrod algebra.
Daisuke Kishimoto informed the author that, in \cite{kaji-2007},  Shizuo Kaji proves that 
 the mod $2$ cohomology 
 of $\Loop BG$ is isomorphic to that of $BG(\mathbb{F}_q)$  as 
 an algebra over the 
mod 2 Steenrod algebra for $G=\spin(n)$ for $n=7,8,9$ and for $G=G_2, F_4$ 
by explicit computation
using twisted cohomology suspension in \cite{kishimoto-kono-2010}.
It is the strong form of Tezuka's conjecture above for these Lie groups. 

The rest of this paper is organized as follows:
In \S2, we prove Proposition~\ref{proposition:ba}.
In \S3,  we recall the cohomology of the free loop space of a space 
whose mod $2$ cohomology is a polynomial algebra. 
In \S4, 
we prove Theorems~\ref{theorem:main2} and \ref{theorem:SO} 
assuming Proposition~\ref{proposition:invariant}.
In \S4, we prove Proposition~\ref{proposition:invariant}
to complete the proof of Theorems~\ref{theorem:main2} and \ref{theorem:SO}.
In \S6, we end this paper by stating a  strong form of Tezuka's conjecture for
$G=\spin(n)$.

The author would like to thank both Daisuke Kishimoto and Shizuo Kaji 
for their communication on their results.
Last but not  least, the author would like to thank the anonymous referee for his/her 
invaluable suggestions which improved the readability of this paper considerably
The referee also pointed out  that
to prove Theorems~\ref{theorem:main1}, \ref{theorem:SO}, 
it is enough to have
the homotopy commutativity $\phi^q \circ \alpha_n \simeq \alpha_n$ instead of the 
strict commutativity $\phi^q\circ \alpha_n=\alpha_n$ in Proposition~\ref{proposition:ba}.
The author appreciates his/her insight.


\section{Proof of Proposition~\ref{proposition:ba}}

In this section, we prove Proposition~\ref{proposition:ba}. We do this in a general setting where 
we deal with a finite subgroup $H$ of a finite Chevalley group $G(\mathbb{F}_q)$.

Let $G$ be a  reductive group scheme over $\Closure$. 
We denote by the $\mathbb{F}_q$-rational points 
$\mathrm{Hom}(\Spec{\mathbb{F}_q}, G)$ by 
$G(\mathbb{F}_q)$. It is a finite group
and the action of the Frobenius map $\phi^q$ is indeed trivial on $G(\mathbb{F}_q)$.
We consider a finite subgroup $D$ of $G(\mathbb{F}_q)$.
As in \cite{friedlander-mislin-1984}, we consider a group scheme $D \times \Spec{\Closure}$ which is just a disjoint union of 
geometric points $\Spec{\Closure}$ parametrized  by  $D$. 
We denote it by 
$D_{\Closure}$. 
We consider the simplicial schemes $BD_{\Closure}$ and
$B(*, D_{\Closure}, D_{\Closure})$ whose $n$-th scheme is the $n$-fold product, $(n+1)$-fold product of $D_{\Closure}$ and structure maps are given in the usual manner, respectively. 
Let $\pi$ be the connected component functor from simplicial schemes to simplicial sets.
Let 
$BD$ be the geometric realization of the simplicial set 
$\pi(BD_{\Closure})$
and let $ED$ be that of the simplicial set $\pi(B(*, D_{\Closure}, D_{\Closure}))$. 
It is clear that $BD=ED/D$ and the Frobenius map $\phi^q$ acts trivially on $BD$.

We need to consider the action of Weyl group $S_n$ on $BA_{n-1}$. 
However, there exists an ambiguity 
on the Weyl group action on classifying spaces in the category of topological spaces.
It is due to the ambiguity of the definition of the classifying space. For a group $H$, some define the 
classifying space 
by considering a contractible space  with free $H$ action. One may define the classifying space of $H$  
to be the quotient space of such space. 
One might define that of $H$  
as  a geometric realization of certain simplicial set obtained from $H$.
These different constructions give spaces with the same homotopy type.
However, 
at the level of topological spaces, they might end up with different topological spaces.
Since we need maps and diagrams in the category of topological space not in their homotopy category, 
we clarify the situation here.
Let $H$ be a subgroup of $D$ and $W$ a subgroup of
the Weyl group $N_D(H)/C_D(H)$ of $H$ in $D$.
Then, we have 
actions of  $W$ on $BH=EH/H$ and on $ED/H$.
On the one hand, the action of $W$  is induced by the action of $W$ on $H$.
Since $W$ acts also on the simplicial scheme $B H_{\Closure}$, it acts on the
geometric realization $BH$ of $\pi(BH_{\Closure})$.
On the other hand, the action of $W$ on  $ED/H$ is given as follows:
For a point $xH$ in $BH=EH/H$, we may define the action of 
$w \in W$ on $ED/H$ by $(xH)w=(xw)H$. 
Consider maps $\alpha_1:BH \to BD$ and $\alpha_2:ED/H \to ED/D=BD$.
Although we have a natural inclusion $i:BH\to ED/H$, 
$w \circ i$ is not equal to $i \circ w$. 
In fact, $(w \circ i )(EH)\not \subset i(EH)$.
It is  clear from the definition  that $\alpha_2 \circ w=\alpha_2$.  
We use $ED/H$ and $\alpha_2$ in the proof of Proposition~\ref{proposition:ba}
instead of  $BH$ and $\alpha_1$.

We consider the group scheme $G(\mathbb{F}_q)_{\Closure}$. 
There exists the evaluation map 
\[
G(\mathbb{F}_q)_{\Closure}=\mathrm{Hom}(\Spec{\mathbb{F}_q}, G) \times 
\Spec{\Closure} \to G.
\]
The geometric realization of the connected components of the simplicial scheme 
$BG(\mathbb{F}_q)_{\Closure}$ is 
the nothing but the classifying space $BG(\mathbb{F}_q)$ of the finite group $G(\mathbb{F}_q)$.
Now, we recall part of the definition of the rigid hypercoverings and etale homotopy type 
 in Friedlander's book \cite{friedlander-1982}
to obtain a map 
\[
BG(\mathbb{F}_q)=|\pi (BG(\mathbb{F}_q)_{\Closure})| \to 
\left| \mathop{\mathrm{holim}}_{\longleftarrow}  (\mathbb{Z}/\ell)_\infty (BG)_{et}\right|.
\]
The etale homotopy type $(BG)_{et}$ is a functor from the rigid hypercoverings 
\[
\pi \circ \Delta:\HyperRigidCover{BG} \to \mbox{(simplicial sets)}
\]
where 
$\HyperRigidCover{BG}$ is a category of certain bisimplicial sets, 
$\Delta$ the diagonal functor and 
$\pi$ connected components functor.
By definition, 
$Y_{\bullet\bullet} \in \HyperRigidCover{BG}$ 
has geometric points $\Spec{\Closure} \to Y_{mn}$ corresponding to geometric points in 
$BG$. 
Since  $G(\mathbb{F}_q)_{\Closure}$ is a disjoint union of 
geometric points, so is $BG(\mathbb{F}_q)_{\Closure}$. Therefore, 
there exists a unique map of simplicial schemes
$BG(\mathbb{F}_q)_{\Closure} \to \Delta(Y_{\bullet\bullet})$.
By taking the connected components, we have a unique map of simplicial sets
\[
\pi(BG(\mathbb{F}_q)_{\Closure}) \to 
 \mathop{\mathrm{holim}}_{\longleftarrow}  (BG)_{et} .
\]
Since the Bousfield-Kan $\mathbb{Z}/\ell$-completion is a functor from simplicial sets to simplicial sets, we also have a map
\[
\pi(BG(\mathbb{F}_q)_{\Closure})  \to  \mathop{\mathrm{holim}}_{\longleftarrow} (\mathbb{Z}/\ell)_\infty (BG)_{et}.
\]
We denote its geometric realization by
\[
\alpha_3:BG(\mathbb{F}_q) \to \left|  \mathop{\mathrm{holim}}_{\longleftarrow} (\mathbb{Z}/\ell)_\infty (BG)_{et} \right|.
\]
It is clear that there holds $\phi^q \circ \alpha_3=\alpha_3$.

Let $H$ be a subgroup of $G(\mathbb{F}_q)$ and $W$ a subgroup of 
$N_{G(\mathbb{F}_q)}(H)/C_{G(\mathbb{F}_q)} (H)$.
Let us consider the classifying space $EG(\mathbb{F}_q)/H$
and the obvious map
\[
\alpha_2:EG(\mathbb{F}_q)/H \to EG(\mathbb{F}_q)/G(\mathbb{F}_q)=BG(\mathbb{F}_q).
\]
Let 
\[
\alpha_0:EG(\mathbb{F}_q)/H \to  \left|  \mathop{\mathrm{holim}}_{\longleftarrow} (\mathbb{Z}/\ell)_\infty (BG)_{et} \right|
\]
be the composition $\alpha_3 \circ \alpha_2$. Then, summing up the above arguments, 
we have the following proposition:

\begin{proposition}
\label{proposition:ba-general}
With the notations above,
there holds
$\alpha_0\circ w=\alpha_0$ for each $w\in W$ and $\phi^q \circ \alpha_0=\alpha_0$
in the category of topological spaces and continuous maps.
\end{proposition}

Let $SO(n)$ be the reductive group scheme over $\Closure$ obtained from the 
complexification of the special orthogonal group $SO(n)$.
Putting $G=SO(n)$, $H=A_{n-1}$, $BA_{n-1}=EG(\mathbb{F}_q)/A_{n-1}$, Proposition~\ref{proposition:ba} immediately 
follows from Proposition~\ref{proposition:ba-general}.


\section{Cohomology of free loop spaces}

In this section, we recall results on the mod $2$ cohomology of the free loop space 
$\Loop BSO_n$.
Let  $X$ be a space with the homotopy type of a connected $CW$ complex 
and it has a base point $*$.
If the mod $2$ cohomology of the  space $X$ is a polynomial algebra, 
the mod $2$ cohomology of the free loop space $\Loop X$ is well-known. 
We refer the reader to 
  Kono and Kishimoto \cite{kishimoto-kono-2010},
 Kuribayashi, Mimura and Nishimoto    \cite{kuribayashi-mimura-nishimoto-2006}.
 In particular, Proposition~3 in \cite{kishimoto-kono-2010} and 
 Theorem~{1.3} in \cite{kuribayashi-mimura-nishimoto-2006}.

 There exists an obvious fibre sequence
\[
\Omega X\stackrel{i}{\longrightarrow} \Loop X \stackrel{\pi}{\longrightarrow} X,
\]
where $\Omega X$ is the based loop space $\pi^{-1}(*)$.
Also, there exists the evaluation map 
$e:S^1 \times \Loop X \to X$. 
Using this evaluation map, we define $\sigma:H^{*}(X) \to H^{*-1}(\Loop X)$ by
\[
e^{*}(x)=1 \otimes x+ u_1 \otimes \sigma(x), 
\]
where $u_1$ is the generator of $H^1(S^1)=\mathbb{Z}/2$.
Then, it is easy to see that $\sigma$ is a derivation in the sense that
\begin{align*}
\sigma(x+y)&=\sigma(x)+\sigma(y), \\
\sigma(x\cdot y)&=\pi^*(x) \cdot \sigma (y)+\sigma(x) \cdot \pi^*(y),
\end{align*}
for $x, y \in H^{*}(X)$.
It is also clear from the definition that it commutes with the 
action of Steenrod squares.
We denote $\sq^{\deg x-1} x$ by $\phi(x)$.
Let us consider the Leray-Serre spectral sequence associated with 
the fibre sequence above.
Denote it by $E_r^{*,*}(\Loop X)$. Then, we have the following theorem.


\begin{theorem}
\label{theorem:polynomial}
Suppose that the mod $2$ cohomology of $X$ is a polynomial algebra, that is,  
$
H^{*}(X)=\mathbb{Z}/2[y_1, \dots, y_n].
$
Then, the Leray-Serre spectral sequence $E_r^{*,*}(\Loop X)$ collapses, 
\[
H^{*}(\Omega X)=\Delta (i^* (\sigma ( y_1)), \dots, i^*(\sigma( y_n)))
\]
and 
\[
H^{*}(\Loop X)=H^{*}(X) \otimes \Delta (\sigma ( y_1), \dots, \sigma( y_n)),
\]
where $\Delta$ is a simple system of generators
and 
$\sigma(y_k)^2=\sigma(\phi(y_k))$ in $H^{*}(\Loop X)$.
\end{theorem}

Theorem~\ref{theorem:polynomial}  is applicable to $BSO(n)$
for all $n$
and $B\spin(n)$ 
for $n<10$.
However, the mod $2$ cohomology of $B\spin(n)$ for $n\geq 10$ 
is no longer a polynomial algebra.
The mod $2$ cohomology of $B\spin_{10}(\mathbb{F}_q)$ is computed 
by Kleinerman in \cite{kleinerman-1982}.
Kuribayashi, Mimura and Nishimoto compute the mod $2$ cohomology of  
the free loop space of $B\spin(10)$ 
and they 
confirm a weak form of Tezuka's conjecture for $\ell=2$, $G=\spin(10)$
in \cite{kuribayashi-mimura-nishimoto-2006}. 
It is the first step towards the computation of the mod $2$ cohomology 
of the
free loop space of $BG$ with the mod $2$ cohomology not isomorphic to a 
polynomial algebra.
However, it seems to be difficult to compute the mod $2$ cohomology of
$\Loop B\spin(n)$  for $n > 10$. 
The proof of Theorem~\ref{theorem:main2}
 in this paper does not depend on the explicit computation of 
the mod $2$ cohomology of $\Loop B\spin(n)$ even in the case $n\leq 10$.


\section{Proof of Theorem~\ref{theorem:main2}}

First, we set up definitions and notations.
By abuse of notation, as for $BSO_n$ in \S1, we denote by 
\[
\pi_n:\Loop B\spin_n \to B\spin_n, \quad \pi_{n,q}:\Loop_{\phi^q}B\spin_n \to B\spin_n,
\]
the projections obtained from the evaluation at $0$, $\mathrm{Map}(I, B\spin_n)\to B\spin_n$, 
by restricting it to subspaces $\Loop B\spin_n$ and $\Loop_{\phi^q} B\spin_n$.
There exists a map
 \[
p_n:B\spin_n \to BSO_n
 \]
 induced by the obvious projection $\spin(n) \to SO(n)$ such that $\phi^q\circ p_n=p_n\circ \phi^q$
 and we have the following commutative diagrams:
 \[
 \begin{diagram}
 \node{\Loop B\spin_n} \arrow{s,l}{\pi_n} \arrow{e,t}{\Loop p_n} \node{\Loop BSO_n} \arrow{s,r}{\pi_n} 
 \\
 \node{B\spin_n} \arrow{e,t}{p_n} \node{BSO_n,} 
 \end{diagram}
 \quad
  \begin{diagram}
 \node{\Loop_{\phi^q} B\spin_{n,q}} \arrow{s,l}{\pi_{n,q}} 
 \arrow{e,t}{\Loop_{\phi^q} p_n} \node{\Loop_{\phi^q} BSO_n} \arrow{s,r}{\pi_{n,q}}
 \\
 \node{B\spin_n} \arrow{e,t}{p_n} \node{BSO_n.} 
 \end{diagram}
 \]
 Let us define spaces $F_n, F_{n, q}$ and maps $\xi_n$, $\tilde{p}_n$, $\xi_{n,q}$, $\tilde{p}_{n,q}$
 by the following pull-back diagrams:
 \[
 \begin{diagram}
 \node{F_{n}} \arrow{e,t}{\tilde{p}_n} \arrow{s,l}{\xi_{n}}  \node{\Loop BSO_n}\arrow{s,r}{\pi_n} \\
 \node{B\spin_n} \arrow{e,t}{p_n} \node{BSO_n,}
 \end{diagram}
 \begin{diagram}
 \node{F_{n,q}} \arrow{e,t}{\tilde{p}_{n,q}} \arrow{s,l}{\xi_{n,q}} \node{\Loop_{\phi^q} BSO_n}\arrow{s,r}{\pi_{n,q}} \\
 \node{B\spin_n} \arrow{e,t}{p_n} \node{BSO_n.}
 \end{diagram}
 \]
By definition, the projections 
\[
\pi_n:\Loop B\spin_n \to B\spin_n, \quad \pi_{n,q}:\Loop_{\phi^q} B\spin_n \to B\spin_n
\]
 factors through $F_n$, $F_{n,q}$. Let 
\[
\eta_n:\Loop B\spin_n \to F_n, \quad \eta_{n,q}:\Loop_{\phi^q} B\spin_n \to F_{n,q},
\]
be maps such that $\pi_n=\xi_n \circ \eta_n$, $\pi_{n,q}=\xi_{n,q}\circ \eta_{n,q}$,
respectively.

Next, we compute the mod $2$ cohomology of $F_n$, $F_{n,q}$ using the results on the mod $2$ cohomology of
$\Loop BSO_n$.


\begin{proposition} \label{proposition:free1}
There exists an isomorphism of algebras over the mod $2$ Steenrod algebra
$$\beta_n: H^{*}(B\spin_n) \otimes_{H^{*}(BSO_n)} H^{*}(\Loop BSO_n)\to H^{*}(F_{n}).$$
\end{proposition}

\begin{proof}
Consider the Eilenberg-Moore spectral sequence associated with the 
fibre square
\[
\begin{diagram}
\node{F_{n}} \arrow{s,l}{\xi_n} \arrow{e,t}{\tilde{p}_n} \node{\Loop BSO_n} \arrow{s,r}{\pi_n} \\
\node{B\spin_n} \arrow{e,t}{p_n} \node{BSO_n.}
\end{diagram}
\]
Since $BSO_n$ is simply-connected, the Eilenberg-Moore spectral sequence
\[
\mathrm{Tor}_{H^{*}(BSO_n)}^{*,*}(H^{*}(B\spin_n), H^{*}(\Loop BSO_n)) \Longrightarrow H^{*}(F_{n})
\]
converges.
By Theorem~\ref{theorem:polynomial}, 
$H^{*}(\Loop BSO_n)$ is a free $H^{*}(BSO_n)$-module.
Hence, 
$$\mathrm{Tor}^{i,*}_{H^{*}(BSO_n)}(H^{*}(B\spin_n), H^{*}(\Loop BSO_n))=0$$ for $i\not=0$.
Therefore, the Eilenberg-Moore spectral sequence collapses at the $E_2$-level and
we have no extension problem and so we obtain 
\begin{align*}
H^{*}(F_n)
&=\mathrm{Tor}^{0,*}_{H^{*}(BSO_n)} (H^{*}(B\spin_n), H^{*}(\Loop B\spin_n))
\\
&=H^{*}(B\spin_n) \otimes_{H^{*}(BSO_n)} H^{*}(\Loop BSO_n). \qedhere
\end{align*}
\end{proof}

Let $E_r^{*,*}(\Loop BSO_n)$,  $E_r^{*,*}(\tilde{B}A_{n-1})$, $E_r^{*,*}(\Loop_{\phi} BSO_n)$ be Leray-Serre spectral sequences converging to $H^{*}(\Loop BSO_n)$, $H^{*}(\tilde{B}A_{n-1})$, $H^{*}(\Loop_{\phi^q} BSO_n)$ associated with
vertical fibrations in the following diagram defined in \S1:
\[
\begin{diagram}
\node{\Loop BSO_n} \arrow{s,l}{\pi_n}  \node{\tilde{B}A_{n-1}} \arrow{s,r}{\pi_0} \arrow{e,t}{\tilde{\alpha}_{n, q}} \arrow{w,t}{\tilde{\alpha}_n} \node{\Loop_{\phi^q}BSO_n} \arrow{s,r}{\pi_{n,q}}
\\
\node{BSO_n} \node{BA_{n-1}} \arrow{e,t}{\alpha_n} \arrow{w,t}{\alpha_n} \node{BSO_n.}
\end{diagram}
\]
The fibre of vertical fibrations is the based loop space $\Omega BSO_n$.
We denote the inclusion of it into $\Loop BSO_n$, $\tilde{B}A_{n-1}$, $\Loop_{\phi^q} BSO_n$ by
\[
i_n:{\Omega BSO_n} \to \Loop BSO_n, \quad i_0:{\Omega BSO_n} \to \tilde{B}A_{n-1}, \quad i_{n,q}:
{\Omega BSO_n} \to \Loop_{\phi^q} BSO_n,
\]
 respectively.


\begin{proposition}
\label{proposition:mono}
Each of the Leray-Serre special sequences above collapses at the $E_2$-level and the induced homomorphisms
\[
E_r^{*,*}(\Loop BSO_n)\stackrel{\tilde{\alpha}_n^*}{\longrightarrow}
 E_r^{*,*}(\tilde{B}A_{n-1})
 \stackrel{\tilde{\alpha}_{n,q}^*}{\longleftarrow}
E_r^{*,*}(\Loop_{\phi^q}  BSO_n)
\]
are monomorphisms for all $r\geq 2$.
\end{proposition}
\begin{proof}
The Leray-Serre spectral sequence $E_r^{*,*}(\Loop BSO_n)$ collapses at the $E_2$-level
by Theorem~\ref{theorem:polynomial}.

Next, we show that the Leray-Serre spectral sequence 
$E_r^{*,*}(\tilde{B}A_{n-1})$ collapses
at the $E_2$-level.
Since the induced homomorphism 
\[
i_n^*=(\tilde{\alpha}_{n} \circ i_0)^*
:
H^{*}(\Loop BSO_n) \to H^{*}({\Omega BSO_n})
\]
 is an epimorphism, $i_0^*$ is also an epimorphism. Hence, 
 by the Leray-Hirsh theorem, the Leray-Serre spectral sequence 
 $E_r^{*,*}(\tilde{B}A_{n-1})$ collapses at the $E_2$-level. 
 Since the induced homomorphism
 \[
 \tilde{\alpha}_n^*:E_2^{*,*}(\Loop BSO_n) \to E_2^{*,*}(\tilde{B}A_{n-1})
 \]
 is equivalent to 
 \[
 \alpha_n^* \otimes 1 : H^{*}(BSO_n) \otimes H^{*}({\Omega BSO_n}) 
 \to 
 H^{*}(BA_{n-1})\otimes H^{*}({\Omega BSO_n}),
 \]
 it  is a monomorphism.
Since both spectral sequence collapse at the $E_2$-level, 
 the induced homomorphism
  \[
 \tilde{\alpha}_n^*:E_r^{*,*}(\Loop BSO_n) \to E_r^{*,*}(\tilde{B}A_{n-1})
 \]
 is also a monomorphism for all $r\geq 2$.

 Finally, we deal with the Leray-Serre spectral sequence 
 $E_r^{*,*}(\Loop_{\phi^q}  BSO_n)$.
The induced homomorphism 
$\tilde{\alpha}_{n,q}^*:E_2^{*,*}(\Loop_{\phi^q} BSO_n) \to 
E_2^{*,*}(\tilde{B}A_{n-1})$ is 
 a monomorphism. 
 By induction on $r$, we show that  $E_r^{*,*}(\Loop_{\phi^q} BSO_n)=E_2^{*,*}(\Loop_{\phi^q} BSO_n)$
 for all $r\geq 2$.
Suppose that we have $E_r^{*,*}(\Loop_{\phi^q} BSO_n)=E_2^{*,*}(\Loop_{\phi^q} BSO_n)$
for some $r$.
 Then, it is clear that the induced homomorphism 
 $\tilde{\alpha}_{n,q}^*:E_r^{*,*}(\Loop_{\phi^q} BSO_n) \to E_r^{*,*}(\tilde{B}A_{n-1})$
 is a monomorphism.
 Since $d_r:E_r^{*,*}(\tilde{B}A_{n-1})\to E_r^{*+r,*-r+1}(\tilde{B}A_{n-1})$ is zero, 
 $\tilde{\alpha}_{n,q}^* \circ d_r= d_r\circ \tilde{\alpha}_{n,q}^*=0$. Hence, we have 
 $d_r=0$ for $E_r^{*,*}(\Loop_{\phi^q} BSO_n)$.
 So, we have 
 \[
 E_{r+1}^{*,*}(\Loop_{\phi^q} BSO_n)=E_r^{*,*}(\Loop_{\phi^q} BSO_n)=E_2^{*,*}(\Loop_{\phi^q} BSO_n).
 \]
Therefore, the Leray-Serre spectral sequence 
 $E_r^{*,*}(\Loop_{\phi^q} BSO_n)$ collapses at the $E_2$-level
 and the induced homomorphism 
 $\tilde{\alpha}_{n,q}^*:E_r^{*,*}(\Loop_{\phi^q} BSO_n) \to E_r^{*,*}(\tilde{B}A_{n-1})$
 is a monomorphism for all $r\geq 2$.
 \end{proof}


To finish the computation of the mod $2$ cohomology of $F_n, F_{n,q}$,
we prove the counter part of Proposition~\ref{proposition:free1}.

\begin{proposition}\label{proposition:free2}
There exists an isomorphism of algebras over the mod $2$ Steenrod algebra
\[
\beta_{n,q}: 
H^{*}(B\spin_n) \otimes_{H^{*}(BSO_n)} H^{*}(\Loop_{\phi^q} BSO_n)\to H^{*}(F_{n,q}).
\]
\end{proposition}

\begin{proof}
By Proposition~\ref{proposition:mono}, 
the Leray-Serre spectral sequence $E_r^{*,*}(\Loop_{\phi^q} BSO_n)$ collapses at the $E_2$-level.
So, $H^{*}(\Loop_{\phi^q} BSO_n)$ is a free $H^{*}(BSO_n)$-module. Therefore, 
as 
in Proposition~\ref{proposition:free1},
the 
Eilenberg-Moore spectral sequence associated with 
the fibre square 
\[
\begin{diagram}
\node{F_{n,q}} \arrow{s,l}{\xi_{n,q}} \arrow{e,t}{\tilde{p}_{n,q}} \node{\Loop_{\phi^q} BSO_n} \arrow{s,r}{\pi_{n,q}} \\
\node{B\spin_n} \arrow{e,t}{p_n} \node{BSO_n}
\end{diagram}
\]
gives us the desired isomorphism.
\end{proof}

With Proposition~\ref{proposition:mono}, we have 
$\tilde{\alpha}_n^*(H^i(\Loop BSO_n))$ and 
$\tilde{\alpha}_{n,q}^*(H^i(\Loop_{\phi^q} BSO_n))$ 
have the same dimension as  vector spaces.
But we do not know whether these  are the same subspace yet.
Next, we show that 
$\tilde{\alpha}_n^*(H^*(\Loop BSO_n))
=\tilde{\alpha}_{n,q}^*(H^*(\Loop_{\phi^q} BSO_n))$ and  prove Theorem~\ref{theorem:SO}.
The key of our proof of Theorem~\ref{theorem:SO} is the following proposition.



\begin{proposition}\label{proposition:invariant}
The induced homomorphism 
\[
\alpha_n^*: H^{*}(BSO_n) \to H^{*}(BA_{n-1})^{S_n}
\]
is an isomorphism
for $n\geq 3$.
\end{proposition}

The induced homomorphism
 $\alpha_n^*$ is a monomorphism
 and
 the image of $\alpha_n^*$ is generated by the Stiefel-Whitney classes 
$w_2, \dots, w_n$.
So, Proposition~\ref{proposition:invariant} is equivalent to say that 
$H^{*}(BA_{n-1})^{S_n}=\mathbb{Z}/2[w_2, \dots, w_n]$ for $n\geq 3$.
It is not difficult to prove Proposition~\ref{proposition:invariant} and 
it might be a folklore but 
the author could not find a reference for Proposition~\ref{proposition:invariant}.
So, for completeness, we give a proof of  this proposition in the next section.
Note that, however, 
\[
H^{*}(BSO_2)=\mathbb{Z}/2[w_2]  \not \simeq H^{*}(BA_1)^{S_2}=H^{*}(BA_1), 
\]
hence we can not prove Theorem~\ref{theorem:SO} for the case $n=2$.


\begin{proof}[Proof of Theorem~\ref{theorem:SO} 
modulo Proposition~\ref{proposition:invariant}]
By Proposition~\ref{proposition:mono}, the induced homomorphisms in Theorem~\ref{theorem:SO} are 
monomorphisms.
Since $\dim H^i(\Loop BSO_n)=
\dim H^i(\Loop_{\phi^q} BSO_n)$ for all $i$, to complete the proof of Theorem~\ref{theorem:SO}, 
it suffices to show that $\dim H^i(\tilde{B}A_{n-1})^{S_n} \leq \dim H^i(\Loop BSO_n)$ for all $i$.
As we stated in \S1, before the statement of Theorem~\ref{theorem:SO}, 
the Weyl group action of $S_n$ on $A_{n-1}$ induces its action on 
the mod $2$ cohomology $H^{*}(\tilde{B}A_{n-1})$.
It also induces its action on the Leray-Serre spectral sequence 
$E_r^{*,*}(\tilde{B}A_{n-1})$.
Since  
the image of $\tilde{\alpha}_{n}^*$ is invariant under the action of the Weyl group $S_n$, 
it acts on $E_r^{0,*}(\tilde{B}A_{n-1})=H^{*}(\Omega BSO_n)$ trivially.
Hence, we have  
\begin{align*}
E_\infty^{*,*}(\tilde{B}A_{n-1})^{S_n}
&=(H^{*}({B}A_{n-1} )
 \otimes H^{*}({\Omega BSO_n}))^{S_n}
  \\
 &
 =H^{*}({B}A_{n-1} )^{S_n} \otimes H^{*}({\Omega BSO_n}).
 \end{align*}
By Proposition~\ref{proposition:invariant}, the above homomorphism $\tilde{\alpha}_{n}^*$  induces the isomorphism
\[
E_\infty^{*,*}(\Loop BSO_n ) \to E_\infty^{*,*}(\tilde{B}A_{n-1})^{S_n}.
\]
In general, for a finite group $W$ and for a finite dimensional filtered $\mathbb{Z}/2[W]$-module $M$,
we have an inequality $\dim M^W \leq \dim ( \mathrm{gr}(M) )^W$, where 
$\mathrm{gr}(M)$ is the associated graded module.
Therefore, we obtain 
\[
\dim H^i(\tilde{B}A_{n-1})^{S_n} \leq 
\dim \mathrm{gr}\, H^i(\tilde{B}A_{n-1})^{S_n}. \]
Since 
\[
E_\infty^{*,*}(\tilde{B}A_{n-1})^{S_n}=E_\infty^{*,*}(\Loop BSO_n)=H^{*}(\Loop BSO_n)
\]
as graded $\mathbb{Z}/2$-modules, 
we have $\dim \mathrm{gr}(H^i(\tilde{B}A_{n-1}))^{S_n} = \dim H^{i}(\Loop BSO_n)$.
Therefore, we have 
\[
\dim H^i(\tilde{B}A_{n-1})^{S_n} \leq 
\dim H^{i}(\Loop BSO_n). 
\]
This completes the proof of Theorem~\ref{theorem:SO} assuming Proposition~\ref{proposition:invariant}.
\end{proof}

Finally, we consider the Gysin sequences associated with $\eta_{n,q}$, $\eta_n$ to
prove Theorem~\ref{theorem:main2}.
The following proposition is immediate from the fact that $BSO_n$ and $B\spin_n$ are
simply-connected and 
that $H^1({\Omega BSO_n})=\mathbb{Z}/2$. But it plays an important role in our proof of Theorem~\ref{theorem:main2}, 
so we emphasize it here.


\begin{proposition}\label{proposition:one}
As vector spaces, 
\[
H^1(F_{n,q})=H^{1}(\Loop_{\phi^q} BSO_n)=H^1(\tilde{B}A_{n-1})^{S_n}= H^1(\Loop BSO_n)=H^1(F_n)= \mathbb{Z}/2
\]
and the following is a chain of isomorphisms\,{\rm :}
\[
H^1(F_{n,q}) \stackrel{\tilde{p}_{n}}{\longleftarrow} 
H^1(\Loop_{\phi^q} BSO_n) \stackrel{\tilde{\alpha}_{n,q}}{\longrightarrow} 
H^1(\tilde{B}A_{n-1})^{S_n} \stackrel{\tilde{\alpha}_{n}}{\longleftarrow}
H^1(\Loop BSO_n) \stackrel{\tilde{p}_{n}}{\longrightarrow}
H^1(F_{n}).
\]
\end{proposition}

Let us consider the commutative diagram below:
\[
\begin{diagram}
\node{\Omega B\spin_n} \arrow{e} \arrow{s,l}{\Omega p_n} 
\node{\Loop_{\phi^q} B\spin_n} \arrow{s,r}{\eta_{n,q}} \arrow{e,t}{\pi_{n,q}}
\node{B\spin_n} \arrow{s,r}{=} 
\\
\node{{\Omega BSO_n}} \arrow{e,t}{j_{n,q}} \node{F_{n,q}}
\arrow{e,t}{\xi_{n,q}} \node{B\spin_n,}
\end{diagram}
\]
where horizontal sequences are fibre sequences and $j_{n,q}$ is the inclusion of the fibre ${\Omega BSO_n}$ into $F_{n,q}$, so that $i_{n,q}=\tilde{p}_{n,q}\circ j_{n,q}$.
The homotopy fibre of $\eta_{n,q}$  and $\Omega p_n:\Omega B\spin_n \to {\Omega BSO_n}$ is homotopy equivalent to $S^0$.
So, we have the Gysin sequence 
\[
\cdots \to H^{*}(\Loop_{\phi^q} B\spin_n) \stackrel{\eta_{n,q}^*}{\longrightarrow}  H^{*}(F_{n,q}) \stackrel{\gamma_{n,q}}{\longrightarrow} H^{*+1}(F_{n,q}) \to \cdots
\]
The homomorphism $\gamma_{n,q}$ is the multiplication by the Euler class $e_{n,q} \in H^{1}(F_{n,q})$.
Comparing the Gysin sequence associated with $\Omega p_n$, 
we have $j_{n,q}^*(e_{n,q})\not=0$.
Since $H^1(F_{n,q})=\mathbb{Z}/2$, $e_{n,q}$ is the  generator of $H^1(F_{n,q})$.
Similarly, 
considering the Gysin sequence associated with $\eta_n:\Loop B\spin_n \to F_{n}$, say
\[
\cdots \to H^{*}(\Loop B\spin_n) \stackrel{\eta_{n}^*}{\longrightarrow}  H^{*}(F_{n}) \stackrel{\gamma_{n}}{\longrightarrow} H^{*+1}(F_{n}) \to \cdots,
\]
the homomorphism $\gamma_n$ is also the multiplication by the Euler class $e_{n} \in H^1(F_n)$ and
the Euler class $e_n \in H^{1}(F_n)$ is also the  generator of $H^1(F_n)=\mathbb{Z}/2$.


\begin{proof}[Proof of Theorem~\ref{theorem:main2}]
Let $e_0$ be the generator of 
$H^1(\tilde{B}A_{n-1})^{S_n}$.
 Also, let $v_{n,q}$, $v_n$ be  generators of $H^{1}(\Loop_{\phi^q} BSO_n)$, $H^1(\Loop BSO_n)$, respectively.
 Then, we have 
 \[
 e_{n,q}=\tilde{p}_{n,q}^*(v_{n,q}), \quad \tilde{\alpha}_{n,q}^*(v_{n,q})=e_0=\tilde{\alpha}_{n}^*(v_n), 
 \quad \tilde{p}_{n}^*(v_{n})=e_{n}
 \]
Let us write 
\[
H^{*}(\Loop BSO_n)
\stackrel{({\tilde{\alpha}_n^*})^{-1}}{\longleftarrow}
H^{*}(\tilde{B}A_{n-1})^{S_n}\stackrel{({\tilde{\alpha}_{n,q}^*})^{-1}}{\longrightarrow}
 H^{*}(\Loop_{\phi^q} BSO_n)
\]
 for the inverses of the isomorphisms in Theorem~\ref{theorem:SO}.
We consider the composition of isomorphisms 
$\beta_{n}$, $\beta_{n,q}$  in Propositions~\ref{proposition:free1},
\ref{proposition:free2}  with
$1 \otimes (\tilde{\alpha}_n^*)^{-1}$,
$1 \otimes (\tilde{\alpha}_{n,q}^*)^{-1}$, respectively.
It is clear that 
$\beta_{n,q}\circ (1 \otimes (\tilde{\alpha}_{n,q}^*)^{-1})$,
$\beta_n \circ (1 \otimes (\tilde{\alpha}_n^*)^{-1})$
are algebra homomorphisms. 
It is also clear that 
 \[
 (\beta_{n,q}\circ (1 \otimes (\tilde{\alpha}_{n,q}^*)^{-1}))(1\otimes e_0)=e_{n,q}
 \]
 and
 \[
(\beta_n \circ (1 \otimes (\tilde{\alpha}_n^*)^{-1}))(1\otimes e_0)=e_{n}.
\]
Therefore, 
the following diagram commutes: 
\[
\begin{diagram}
\node{H^{*}(F_{n,q})} \arrow{e,t}{\gamma_{n,q}} \node{H^{*+1}(F_{n,q})=
\Sigma(H^*(F_{n,q}))} \\
\node{H^{*}(B\spin_n)\otimes_{H^{*}(BSO_n)} H^{*}(\tilde{B}A_{n-1})^{S_n}} 
 \arrow{n,lr}{\simeq}{\beta_{n,q}\circ (1 \otimes (\tilde{\alpha}_{n,q}^*)^{-1})}
\arrow{s,lr}{\simeq}{\beta_n \circ (1 \otimes (\tilde{\alpha}_n^*)^{-1})}\arrow{e}
\node{\Sigma( H^{*}(B\spin_n)\otimes_{H^{*}(BSO_n)} H^{*}(\tilde{B}A_{n-1})^{S_n}) }
 \arrow{s,lr}{\simeq}{\Sigma (\beta_n \circ (1 \otimes (\tilde{\alpha}_n^*)^{-1})) }
 \arrow{n,lr}{\simeq}{\Sigma (\beta_{n,q}\circ (1 \otimes (\tilde{\alpha}_{n,q}^*)^{-1}))}
\\
\node{H^{*}(F_n)}  \arrow{e,t}{\gamma_n} \node{H^{*+1}(F_n)
=\Sigma(H^{*}(F_{n})),}  
\end{diagram}
\]
where horizontal homomorphisms are the multiplication by $e_{n,q}, 1\otimes e_0 , e_n$
and $\Sigma^n$ is the degree shift functor such that $(\Sigma^n M)^i=M^{i+n}$.
Thus, 
we obtain isomorphisms of graded $\mathbb{Z}/2$-modules
\begin{align*}
H^{*}(\Loop_{\phi^q} B\spin_n) &= \mathrm{Ker}\, {\gamma_{n,q}} \oplus \Sigma^{-1} \mathrm{Coker}\, {\gamma_{n,q}} \\
&= \mathrm{Ker}\, \gamma_n \oplus \Sigma^{-1} \mathrm{Coker}\, \gamma_n \\
&=H^{*}(\Loop B\spin_n).
\end{align*}
This completes the proof of Theorem~\ref{theorem:main2}.
\end{proof}


\section{Proof of Proposition~\ref{proposition:invariant}}

In order to prove Proposition~\ref{proposition:invariant}, we recall the invariant theory of finite groups. We refer the reader to Smith's book \cite{smith-1995}, for instance.
Let $K$ be a ground field. 
Let $W$ be a finite group acting on $K$-vector space $V$. We denote by 
$V^*$ the dual of $V$ with respect to a basis $v_1, \dots, v_k$.
We denote its dual basis $x_1, \dots, x_k$.
Then the finite group $W$ acts  on a polynomial algebra $K[V]
=K[x_1, \dots, x_k]$. 
The following proposition is Proposition~{5.5.5} in Smith's book \cite{smith-1995}.


\begin{proposition}\label{proposition:smith}
Suppose that $W$ acts on $V$ faithfully. 
Suppose that $\dim V=k$.
Let $R$ be a subalgebra generated by homogeneous polynomials 
$f_1, \dots, f_{k}$ and $R\subset K[V]^{W}$. If the inclusion $R\to K[V]$ is finite, 
and
if $\deg f_1 \cdots \deg f_{k}=|W|$, then $R$ is a polynomial algebra and 
$$R=K[V]^{W}.$$
\end{proposition}

Now, let $K=\mathbb{Z}/2$ and consider the action of 
$S_n$ on $A_{n-1}$ and $H^1(BA_{n-1})$.
The group $A_{n-1}$ is the diagonal matrixes  in $SO(n)$ 
whose diagonal entries are $\pm 1$.
Let $A_{n}$ be the subgroup of $O(n)$ consisting of 
diagonal matrixes whose diagonal entries are $\pm 1$.
$S_n$ acts on $A_{n}$ by permutation of diagonal entries.
Then, $H^{*}(BA_n)=\mathbb{Z}/2[x_1, \dots, x_n]$ and
$H^{*}(BA_n)^{S_n}=\mathbb{Z}/2[w_1, \dots, w_n]$, where 
$w_i$ is the $i$-th symmetric function of $x_1, \dots, x_n$
and there exists an obvious $S_n$-equivariant epimorphism 
\[
H^{*}(BA_n) \to H^{*}(BA_{n-1})
\]
induced by the inclusion of $A_{n-1}$ into $A_n$.
Moreover, we have 
\[
H^{*}(BA_{n-1})=\mathbb{Z}/2[x_1, \dots, x_n]/(x_1+\cdots+x_n).
\]
See Corollary 4.19, p.135, in the book of Mimura and Toda \cite{mimura-toda-book}.
It is clear that $H^{*}(BA_n)^{S_n}\to H^{*}(BA_n)$ is finite.
Therefore, the composition 
\[
H^{*}(BA_n)^{S_n}\to H^{*}(BA_n) \to H^{*}(BA_{n-1})
\]
is also finite.
Since $w_1$ maps to $0$, the inclusion of the subalgebra  of $H^{*}(BA_{n-1})$ 
generated by $w_2, \dots, w_n$ into $H^{*}(BA_{n-1})$ is also finite.
It is also clear that $\deg w_2 \times \cdots \times \deg w_{n}=n!=|S_n|$.
Therefore, if the action of $S_n$ is faithful on $H^1(BA_{n-1})$, by Proposition~\ref{proposition:smith} above,
we have 
\[
H^{*}(BA_{n-1})^{S_n}=\mathbb{Z}/2[w_2, \dots, w_n].
\]


So, it remains to see that the action of $S_n$ on $H_1(BA_{n-1})$ is faithful,
or equivalently that its action on 
$H^1(BA_{n-1})$ is 
faithful.


\begin{proposition}\label{proposition:faithful}
For $n \geq 3$, $S_n$ acts on $H^1(BA_{n-1})$ faithfully.
\end{proposition}

\begin{proof} 
Let $S_{n-1}$ be the subgroup consisting of permutations of $x_1, \dots, x_{n-1}$.
 Let $\sigma(x_i)=x_{i+1}$ for $i=1, \dots, n-2$ and $\sigma(x_{n-1})=x_n=x_1+\cdots+x_{n-1}$.
It is clear that $\sigma^{n}(x_i)=x_i$ for $i=1, \dots, n-1$. Thus, $\sigma^{n}=1$.
 Moreover, $\sigma^j(x_{n-j})=x_1+\cdots+x_{n-1}$, hence $\sigma^j  \not \in S_{n-1}$
  for $j=1, \dots, n-1$. 
  In order to prove that the action of $S_n$ on $H^1(BA_{n-1})$ is faithful, it suffices to show that
   $\sigma$ and $S_{n-1}$ generate a subgroup of $GL_{n-1}(\mathbb{Z}/2)$ 
   whose order is greater than or equal to $n!$.
 Let $\tau \in S_{n-1}$.
 Then, 
 $\tau \sigma^j(x_{n-j})=x_1+\cdots+x_{n-1}$ and $\tau \sigma (x_i)=x_{k}$ for some 
 $k$  for each $i \not = n-j$. 
 Therefore, $S_{n-1} \cap   S_{n-1}\sigma^j =\emptyset$ for $1\leq j\leq n-1$.
 Hence, $S_{n-1}\sigma^i \cap   S_{n-1}\sigma^{k} =\emptyset$ for $0\leq i<k\leq n-1$.
So, $\sigma$ and $S_{n-1}$ generate
 a subgroup of order greater than or equal to $S_{n}$. Hence, the action of $S_{n}$ on 
 $H^1(BA_{n-1})$ is faithful.
\end{proof}


\section{Conjecture}


We end this paper  with the following conjecture similar to Theorem~\ref{theorem:SO}
for $B \spin_n$.
Consider  the subgroup 
\[
\bar{A}_{n-1}=p_{n}^{-1}(A_{n-1})
\]
 in $\spin(n)$, where 
$p_n:\spin(n)\to SO(n)$ is the obvious projection and $A_{n-1}$ is the elementary abelian $2$-subgroup 
consisting of diagonal matrices whose non-zero entries are $\pm 1$ as in \S1.
Let $\alpha_n:B\bar{A}_{n-1} \to \Loop B\spin_n$ be the map induced by the inclusion
of $\bar{A}_{n-1}$ into $\spin(n)$.
We define  maps
\[
\tilde{\alpha}_n:\tilde{B}\bar{A}_{n-1}
\to \Loop B\spin_n, \quad \tilde{\alpha}_{n,q}:\tilde{B}\bar{A}_{n-1}
\to \Loop_{\phi^q} B\spin_n
\] by the following pull-back diagram:
\[
\begin{diagram}
\node{\Loop B\spin_n} \arrow{s,l}{\pi_{n}}
\node{\tilde{B}\bar{A}_{n-1}} 
\arrow{e,t}{\tilde{\alpha}_{n,q}}  
\arrow{s,r}{\pi_0}
\arrow{w,t}{\tilde{\alpha}_n}
\node{\Loop_{\phi^q} B\spin_n} \arrow{s,r}{\pi_{n,q}}
\\
\node{B\spin_n}
\node{B\bar{A}_{n-1}} 
\arrow{e,t}{\alpha_n} 
\arrow{w,t}{\alpha_n}
\node{B\spin_n.}
\end{diagram}
\]

\begin{conjecture}
Let $q$ be a power of an odd prime $p$ and $\mathbb{F}_q$
 the finite field of $q$ elements.
For $n\geq 3$, 
the induced homomorphisms
\[
H^{*}(\Loop B\spin_n) \stackrel{\tilde{\alpha}_n^*}{\longrightarrow}
 H^{*}(\tilde{B}\bar{A}_{n-1})
 \stackrel{\tilde{\alpha}_{n,q}^*}{\longleftarrow}
H^{*}(\Loop_{\phi^q}B\spin_n)
\]
are monomorphisms
and that
the image of $\tilde{\alpha}_{n,q}^*$ is  the same to that of $\tilde{\alpha}_n^*$.
\end{conjecture}

If it holds, this conjecture implies that the mod $2$ cohomology of 
the free loop space of the classifying space $B\spin(n)$
 is isomorphic to 
the mod $2$ cohomology of the finite Chevalley group 
$\spin_n(\mathbb{F}_q)$
not only as a ring but also as an algebra over the mod $2$ 
Steenrod algebra, that is, a  strong form of Tezuka's conjecture.


\begin{thebibliography}{9}

\bibitem{atiyah-bott-1982}
M. F. Atiyah\ and\ R. Bott, The Yang--Mills equations over Riemann surfaces, Philos. Trans. Roy. Soc. London Ser. A {\bf 308} (1983), no.~1505, 523--615. MR0702806 (85k:14006)

\bibitem{friedlander-1982}
E. M. Friedlander, {\it \'Etale homotopy of simplicial schemes}, Annals of Mathematics Studies, 104, Princeton Univ. Press, Princeton, NJ, 1982. MR0676809 (84h:55012)

\bibitem{friedlander-mislin-1984}
E. M. Friedlander\ and\ G. Mislin, Cohomology of classifying spaces of complex Lie groups and related discrete groups, Comment. Math. Helv. {\bf 59} (1984), no.~3, 347--361. MR0761803 (86j:55011)


\bibitem{kishimoto-kono-2010}
D. Kishimoto\ and\ A. Kono, On the cohomology of free and twisted loop spaces, J. Pure Appl. Algebra {\bf 214} (2010), no.~5, 646--653. MR2577671 (2011a:55010)


\bibitem{kaji-2007}
S. Kaji, Mod 2 cohomology of 2-compact groups of low rank, J. Math. Kyoto Univ. {\bf 47} (2007), no.~2, 441--450. MR2376966 (2009c:55017a)


\bibitem{kleinerman-1982}
S. N. Kleinerman, The cohomology of Chevalley groups of exceptional Lie type, Mem. Amer. Math. Soc. {\bf 39} (1982), no.~268, {\rm viii}+82 pp. MR0668808 (84i:20048)


\bibitem{kuribayashi-mimura-nishimoto-2006}
K. Kuribayashi, M. Mimura\ and\ T. Nishimoto, Twisted tensor products related to the cohomology of the classifying spaces of loop groups, Mem. Amer. Math. Soc. {\bf 180} (2006), no.~849, vi+85 pp. MR2203859 (2006k:55032)



\bibitem{mimura-toda-book}
M. Mimura\ and\ H. Toda, {\it Topology of Lie groups. I, II}, translated from the 1978 Japanese edition by the authors, Translations of Mathematical Monographs, 91, Amer. Math. Soc., Providence, RI, 1991. MR1122592 (92h:55001)


\bibitem{quillen-1972}
D. Quillen, On the cohomology and $K$-theory of the general linear groups over a finite field, Ann. of Math. (2) {\bf 96} (1972), 552--586. MR0315016 (47 \#3565)

\bibitem{smith-1995}
L. Smith, {\it Polynomial invariants of finite groups}, Research Notes in Mathematics, 6, A K Peters, Wellesley, MA, 1995. MR1328644 (96f:13008)


\bibitem{tezuka-1998}
M. Tezuka, 
On the cohomology of finite Chevalley groups and free loop spaces of classifying spaces, 
Suueikenkoukyuuroku {\bf 1057} (1998), 54-55.
\end{thebibliography}
 \end{document}